\documentclass{amsart}

\usepackage{amssymb}
\usepackage{amsthm}
\usepackage{amsfonts}
\usepackage{mathrsfs}
\usepackage{amsmath}
\usepackage{extarrows}
\newtheorem{theorem}{Theorem}[section]
\newtheorem{corollary}[theorem]{Corollary}
\newtheorem{lemma}[theorem]{Lemma}
\newtheorem{proposition}[theorem]{Proposition}

\numberwithin{equation}{section}

\newcommand{\RePt}{\mathrm{Re}\,}
\newcommand{\ImPt}{\mathrm{Im}\,}
\newcommand{\ball}{\mathbb{B}}

\newcommand{\bfL}{\mathcal{L}}

\newcommand{\calL}{\mathcal{L}}
\newcommand{\calU}{\mathcal{U}}

\newcommand{\bfi}{\mathbf{i}}
\newcommand{\bfbeta}{\boldsymbol{\beta}}
\newcommand{\bfrho}{\boldsymbol{\rho}}
\newcommand{\calB}{\mathcal{B}}
\newcommand{\bbB}{\mathbb{B}}
\newcommand{\bbC}{\mathbb{C}}
\newcommand{\bbR}{\mathbb{R}}
\newcommand{\bbH}{\mathbb{H}}
\newcommand{\bbN}{\mathbb{N}}

\newcommand{\bSieg}{b\mathcal{U}}

\begin{document}

\title[A Bloch-type space]{A Bloch-type space and the predual of  $A_\lambda^1$\\ on the Siegel upper half-space}

\author{Congwen Liu}
\email{cwliu@ustc.edu.cn}
\address{CAS Wu Wen-Tsun Key Laboratory of Mathematics,
School of Mathematical Sciences,
University of Science and Technology of China\\
Hefei, Anhui 230026,
People's Republic of China}
\thanks{
The first author was supported by the National Natural Science Foundation of China grant 11971453.}

\author{Jiajia Si}
\email{sijiajia@mail.ustc.edu.cn}
\address{School of Mathematics and Statistics, Hainan University, Haikou, Hainan 570228,
People's Republic of China.}
\thanks{The second author was supported by the National Natural Science Foundation of China grant 12201159, the Hainan Provincial Natural Science Foundation of China grant 124YXQN413 
and the Nanhai New Star Project grant 2025NHXX203.}


\begin{abstract}
This paper aims to determine the predual of the Bergman space $A_\lambda^1$ on the Siegel upper half-space. To achieve this, a Bloch-type space $\widetilde{\calB}$ is introduced and studied, and some of its essential properties are established. We identify the little Bloch-type space $\widetilde{\calB}_0$ with the predual of $A_\lambda^1$ via a duality pairing.
\end{abstract}


\maketitle

\section{introduction}

Let $\mathbb{C}^n$ be the $n$-dimensional complex Euclidean space. The Siegel upper half-space of  $\mathbb{C}^n$ is the set
\[
\calU=\left\{ z\in\mathbb{C}^n:\bfrho(z)>0\right\},
\]
where $\bfrho(z)=\ImPt z_n-|z^{\prime}|^2$. 
Here and throughout the paper, we use the notation
\[
z=(z^{\prime},z_n),\quad \text{where}\, z^{\prime}=(z_1,\cdots,z_{n-1})\in\mathbb{C}^{n-1}\,\, \text{and}\,\, z_n\in\mathbb{C}^{1}.
\]

For $0<p<\infty$ and $\lambda>-1$, the weighted Bergman space $A_{\lambda}^p(\calU)$ consists of holomorphic functions $f$ on $\calU$ for which
\begin{equation*}
\|f\|_{p,\lambda}=\bigg\{\int\limits_{\calU} |f(z)|^p dV_{\lambda}(z)\bigg\}^{1/p}<\infty,
\end{equation*}
where $dV_{\lambda}(z)=c_{\lambda} \bfrho(z)^{\lambda} dV(z)$,
\[
 c_{\lambda} =\frac{\Gamma(n+1+\lambda)}{4\pi^n \Gamma(1+\lambda)},
\]
 and $dV$ is the Lebesgue measure on $\mathbb{C}^n$.
It is well known that there is an orthogonal projection of $L^2(\calU,dV_{\lambda})$ onto $A_{\lambda}^2(\calU)$, which can be expressed as an integral operator:
\[
\mathcal{P}_{\lambda}f(z) =  \int\limits_{\calU} K_{\lambda}(z,w) f(w) dV_{\lambda}(w),
\]
where
\[
K_{\lambda}(z,w)=\left[ \frac{i}{2}(\overline{w}_n-z_n)-z^{\prime} \cdot \overline{w^{\prime}} \right]^{-(n+1+\lambda)}.
\]
See, for instance, \cite[Theorem 5.1]{Gin64}.
For brevity, set
\[
\bfrho(z,w)=\frac{i}{2}(\overline{w}_n-z_n)-z^{\prime} \cdot \overline{w^{\prime}}.
\]
Note that $\bfrho(z)=\bfrho(z,z)$. Thus, the kernel $K_{\lambda}(z,w)$ can be rewritten as
\[
K_{\lambda}(z,w)=\frac {1}{\bfrho(z,w)^{n+1+\lambda}}.
\]

Set $\bfi=(0^{\prime},i)$. Let
\[
\widetilde{K_{\lambda}}(z,w) = K_{\lambda}(z,w)-K_{\lambda}(\bfi, w),
\]
and its corresponding projection be
\[
\widetilde{\mathcal{P}_{\lambda}} f(z)= \int\limits_{\calU} \widetilde{K_{\lambda}} (z,w) f(w) dV_\lambda(w).
\]
In the one-dimensional case, the upper half-plane, Coifman and Rochberg \cite{CR80} briefly showed that the dual of Bergman space $A^1$ can be realized as the projection of $L^{\infty}$ under $\widetilde{\mathcal{P}_{\lambda}}$. B\'ekoll\'e \cite{Bek83,Bek86} successively generalized this result to the setting of the Siegel upper half-space and the symmetric Siegel domain of type II, 
by introducing a notion of Bloch-type space as a bridge. Naturally, one might ask what is the predual of $A^1$ on these domains.
This paper is devoted to treating this problem in the setting of the Siegel upper half-space.
To this end,  a type of Bloch space is introduced and sdudied. 

Let $b\calU=\{z\in \mathbb{C}^n: \bfrho(z)=0\}$ be the boundary of $\calU$. Clearly,
\[
\bfL_j := \partial_j+ 2i\bar{z}_j \partial_n\quad (j=1,\ldots, n-1)
\]
forms a global basis for the space of tangential $(1, 0)$ vector fields on the boundary $b\calU$.
Write $\bfL_n := \partial_n$.
In \cite{Si22BMO} a notion of Bloch space on $\calU$ is introduced in a natural way, denoted by $\calB$,  consisting of holomorphic functions $f$ such that
\[
\| f \|_{\calB}:=\sup_{z\in\calU} |\widetilde{\nabla}f(z)| <\infty,
\]
where
\begin{equation*}\label{eqn:inv.gra.f}
|\widetilde{\nabla} f| = \Bigg(2\bfrho \sum_{j=1}^{n-1} |\bfL_j f|^2 + 8\bfrho^2 |\bfL_n f|^2\Bigg)^{1/2}.
\end{equation*} 
$|\widetilde{\nabla}|$ is usually referred to the invariant gradient since it is invariant under the  automorphism group $\text{Aut}(\calU)$;  see for instance \cite[Chapter 3]{Sto199}.
It was shown in \cite{Si22BMO} that the semi-norm $\|\cdot\|_{\calB}$ is equivalent to the Bloch semi-norm on the unit ball. Consequently, $\calB$ is isomorphic to the Bloch space on the unit ball and shares many of its functional properties.
Note that $\widetilde{\mathcal{P}_{\lambda}}f(\bfi)=0 $ for appropriate functions $f$.
It leads to a notion of Bloch-type space, denoted by $\widetilde{\calB}$, the subspace of $\calB$ such that functions vanish at $\bfi$. 

The key tool of the paper is the following integral representation. 
To our knowledge, similar results have not been found in the literature of holomorphic functions in the unbounded domains. 
Let $\mathbb{N}_0$ denote the set of nonnegative integers, 
and for a multi-index $\alpha\in \mathbb{N}_0^n$, set $\bfL^{\alpha}=(\bfL_1)^{\alpha_1} \cdots (\bfL_n)^{\alpha_n}$.

\begin{theorem}\label{thm:reproducingfml}
If $f\in \widetilde{\calB}$, then
\begin{equation*}\label{eqn:reproducingfml}
f= \frac{(-2i)^N\Gamma(1+\lambda)}{\Gamma(1+\lambda+N)} \widetilde{\mathcal{P}_{\lambda}} (\bfrho^N \bfL_n^N f)
\end{equation*}
for any $N\in\bbN_0$ and $\lambda>-1$.
\end{theorem}

By Theorem \ref{thm:reproducingfml}, to our surprise, the space $\widetilde{\calB}$ is identical to a type of Bloch space introduced by B\'ekoll\'e in \cite{Bek83}.
Let $\calB_*$ denote the space of holomorphic functions satisfying 
\[
\sup_{z\in\calU} \bfrho(z) |\bfL_n f(z)| <\infty,
\]
and let $\mathcal{N}$ be the subspace annihilated by $\bfL_n$.
He showed that $\widetilde{\mathcal{P}_{0}}(L^{\infty}(\calU))=\calB_*/\mathcal{N}$.
Consequently, $\widetilde{\calB}$ and $\calB_*/\mathcal{N}$ coincide (see Proposition \ref{cor:BPL}). 
However, B\'ekoll\'e only conducted a glance at this space.
One of the contributions of the paper is to improve the research.
Moreover, the advantage of the notion of $\widetilde{\calB}$ is that it provides a convenient way to analyze its structure, as we shall see.

\begin{theorem}\label{thm:bloch norm equiv}
For any positive integer $N$,
\[
\|f\|_{\calB} \approx \|\bfrho^N \bfL_n^N f\|_{\infty} \approx \sum_{|\alpha | = N} \| \bfrho^{\langle \alpha \rangle} \bfL^{\alpha} f\|_{\infty}
\]
as $f$ ranges over $\widetilde{\calB}$. Here $\langle \alpha \rangle=|\alpha^{\prime}|/2+\alpha_n$. 
\end{theorem}

Here and throughout the paper, we abbreviate inessential constants involved in inequalities by writing $A\lesssim B$ for positive quantities $A$ and $B$ if the ratio $A/B$ has a positive upper bound. Also, $A\approx B$ means both $A\lesssim B$ and $B\lesssim A$.

Theorem \ref{thm:bloch norm equiv} indicates that the space $\calB$ can be tough of as the limit case of the Bergman spaces $A_{\lambda}^p(\calU)$ as $p\to\infty$; 
see an analogue for Bergman space in \cite[Theorem 1.4]{LSX22}. 
For  harmonic Bloch functions on the upper half-space of $\mathbb{R}^n$, a similar result was obtained by Ramey and Yi \cite[Theorem 5.13]{RY96}. 

Finally, we give the characterization of the predual of $A_\lambda^1(\calU)$.
Let $\partial \widehat{\calU}:=b\calU\cup\{\infty\}$ and $C_0(\calU)$ be the space of complex-valued continuous functions $f$ in $\calU$ such that $f(z)\to 0$ as $z\to\partial \widehat{\calU}$. 
The little Bloch-type space $\widetilde{\calB}_0$ is the subspace of $\widetilde{\calB}$ such that $|\widetilde{\nabla}f|\in C_0(\calU)$.

\begin{theorem}\label{thm:predual}
Suppose $\lambda>-1$. The predual of $A_\lambda^1(\calU)$ is isomorphic to $\widetilde{\calB}_0$ (with equivalent norms) under the duality pairing
\[
\langle f,g \rangle_\lambda = \int\limits_{\calU} \bfrho(z) \calL_n f(z) \overline{g(z)}  dV_\lambda(z)
\]
for $f$ in $\widetilde{\calB}_0$ and $g$ in $A_\lambda^1(\calU)$.
\end{theorem}

The paper is organized as follows. Section 2 presents auxiliary results. 
Section 3  includes the cancellation property of $A_{\lambda}^1(\calU)$.
Section 4 studies the projection $\widetilde{\mathcal{P}_{\lambda}}$. 
Section 5 is devoted to proving Theorem \ref{thm:reproducingfml}.
Theorem \ref{thm:bloch norm equiv} is proved in Section 6 .
Finally, our main purpose, Theorem \ref{thm:predual} is established in Section 7.


\section{Preliminaries}

\subsection{Elementary results}

\begin{lemma}[{\cite[Theorem 2.1]{DK93}}]\label{thm:DK}
Suppose that $1\leq p < \infty$, $t>-1$ and $\lambda\in \mathbb{R}$ satisfy
\[
\begin{cases}
\lambda > \frac {t+1}{p}-1,& 1<p<\infty,\\
\lambda \geq t,& p=1.
\end{cases}
\]
If $f\in A_{t}^p(\calU)$ then $f=\mathcal{P}_{\lambda} f$.
\end{lemma}

\begin{lemma}[{\cite[Lemma 13]{Liu18}}]
Let $\theta>0$ and $\gamma>-1$. The identities
\begin{equation}\label{eqn:keylem1}
\int\limits_{b\calU} \frac {d\bfbeta(u)} {|\bfrho(z,u)|^{n+\theta}} =
\frac {4\pi^{n} \Gamma(\theta)} {\Gamma^2\left(\frac {n+\theta}{2}\right)}\ \bfrho(z)^{-\theta}
\end{equation}
and
\begin{equation}\label{eqn:keylem}
\int\limits_{\calU} \frac {\bfrho(w)^{\gamma} dV(w)} {|\bfrho(z,w)|^{n+1+\theta+\gamma}}  =
\frac {4 \pi^{n} \Gamma(1+\gamma) \Gamma(\theta)} {\Gamma^2\left(\frac {n+1+\theta+\gamma}{2}\right)}\ \bfrho(z)^{-\theta}
\end{equation}
hold for all $z\in \calU$.
\end{lemma}

\begin{lemma}[{\cite[Lemma 2.7]{LS20}}]\label{blem}
Given $r > 0$, the inequalities
\begin{equation}\label{eqn:eqvltquan}
\frac {1-\tanh (r)}{1+ \tanh (r)} \leq \frac{|\bfrho(z,u)|}{|\bfrho(z,v)|}
\leq \frac {1+\tanh (r)}{1-\tanh (r)}
\end{equation}
hold for all $z,u,v\in \calU$ with $\beta(u,v)\leq r$, where $\beta(u,v)$ is the Bergman metric on $\calU$ given by 
\begin{equation*}\label{eqn:hyperdist}
 \beta(u,v)=\tanh^{-1}\sqrt{1-\frac{\bfrho(u)\bfrho(v)}{|\bfrho(u,v)|^2}}.
\end{equation*}
\end{lemma}

Note that $z\to \partial\widehat{\calU}$ if and only if $\beta(z,\bfi)\to\infty$.
Indeed, first note that 
\begin{equation}\label{eqn:rho(z,i)}
\sqrt{|z|^2+1}\leq2|\bfrho(z,\bfi)|\leq |z|+1\quad \text{and} \quad \bfrho(z)\leq |z|.
\end{equation}
Then we have
\[
4\bfrho(z)/(|z|+1)^2 \leq \bfrho(z)/|\bfrho(z,\bfi)|^2 \leq 4\min\{\bfrho(z), |z|/(|z|^2+1) \}.
\]
It follows that $z\to \partial\widehat{\calU}$ if and only if $\bfrho(z)/|\bfrho(z,\bfi)|^2\to 0$, as desired.

Let $\bbB$ denote the unit ball of $\bbC^n$. Recall that $\bbB$ and $\calU$ are biholomorphic equivalent through the Cayley transform $\Phi:\bbB \to \calU$  given by
\[
(z^{\prime}, z_{n})\; \longmapsto\; \left( \frac {z^{\prime}}{1+z_{n}},
i\left(\frac {1-z_{n}}{1+z_{n}}\right) \right).
\]
We refer to  \cite[Chapter XII]{Ste93} for the following properties of Cayley transform.

\begin{lemma}\label{lem:cayley}
The Cayley transform $\Phi$ has the following elementary properties:
\begin{enumerate}
\item[(i)]
The identity
\begin{equation*}\label{eqn:identity14phi}
\bfrho(\Phi(\xi),\Phi(\eta)) = \frac {1-\xi\cdot \overline{\eta}} {(1+\xi_{n}) (1+\overline{\eta}_{n})}
\end{equation*}
holds for all $\xi,\eta\in \ball$.
\item[(ii)]
The real Jacobian of $\Phi$ at $\xi\in \ball$ is
\begin{equation*}\label{eqn:jacobian4phi}
\left(J_{R}\Phi\right)(\xi) = \frac {4}{|1+\xi_{n}|^{2(n+1)}}.
\end{equation*}
\end{enumerate}
\end{lemma}

\subsection{Heisenberg Group}
We briefly recall several basic facts about the Heisenberg group that can be found in \cite[Chapter XII]{Ste93}.
The Heisenberg group $\bbH$ is the set
\[
\bbC^{n-1} \times \bbR = \{ [\zeta,t]: \zeta\in \bbC^{n-1}, t\in \bbR\}
\]
endowed with the group operation
\[
[\zeta,t]\cdot [\eta,s]=[\zeta+\eta, t+s+2 \mathrm{Im}(\zeta\cdot \bar{\eta})].
\]
The identity element is $[0,0]$ and the inverse of
$[\zeta,t]$ is $[\zeta,t]^{-1}=[-\zeta,-t]$.
The measure $dh$ on $\bbH$ is the usual Lebesgue measure $d\zeta dt$ on $\bbC^{n-1}\times \bbR$, here we write $h=[\zeta,t]$; it is Haar measure for $\bbH$.

To each element $h=[\zeta,t]$ of $\bbH$, we associate the following (holomorphic) affine self-mapping of
$\calU$:
\begin{equation*}\label{eqn:groupaction}
h:\; (z^{\prime},z_{n}) \longmapsto (z^{\prime}+\zeta, z_{n}+t+ 2i z^{\prime}\cdot \bar{\zeta} + i|\zeta|^2).
\end{equation*}
These mappings are simply transitive on the boundary $b\calU$,
so we can identify the Heisenberg group with $b\calU$ via its action on the origin
\[
\bbH \ni [\zeta,t] \,\longmapsto\, (\zeta, t+i|\zeta|^2) \in b\calU.
\]
This identification allows us to transport the Haar measure $dh$ on $\bbH$ to the measure $d\bfbeta$ on $b\calU$;
that is, we have the integration formula
\begin{equation*}\label{eqn:defnofbeta}
\int\limits_{b\calU} f d\bfbeta = \int\limits_{\bbH}f(h(0)) dh
\end{equation*}
for suitable functions $f$.
Moreover, $d\bfbeta$ is $\bbH$-invariant; that is, $d\bfbeta(h(z))=d\bfbeta(z)$ for each $h\in\bbH$.
Finally, for every $h\in\bbH$ and $z\in\calU$, $u\in b\calU$, one has
\[
\bfrho(h(z),h(u))=\bfrho(z,u).
\]

\section{Cancellation property of $A_{\lambda}^1(\calU)$}

\begin{lemma}\label{lem:bUrho=0}
Let $s>0$. Then, for every $z\in\calU$,
\begin{equation*}\label{eq:int,rho}
\int\limits_{\bSieg} \frac {d\bfbeta(u)} {\bfrho(u,z)^{n+s}} = 0.
\end{equation*}
\end{lemma}

\begin{proof}
Fix $z\in\calU$ and set $h_z=[-z',-\RePt z_n]\in\bbH$. One easily checks that $h_z(z) = \bfrho(z) \bfi$.
Using the change of variables $u\mapsto h_z^{-1}(u)$ and the invariance of $d\bfbeta$, we obtain
\begin{align*}\label{eqn:cancl1}
\int\limits_{\bSieg} &\frac {d\bfbeta(u)} {\bfrho(u,z)^{n+s}}
= \int\limits_{\bSieg} \frac {d\bfbeta(u)} {\bfrho(h^{-1}_z(u), z)^{n+s}}
= \int\limits_{\bSieg} \frac {d\bfbeta(u)} {\bfrho(u, h_z(z))^{n+s}}\\
&=  \int\limits_{\bSieg} \frac {d\bfbeta(u)} {\bfrho(u, \bfrho(z)\bfi)^{n+s}}
= \int\limits_{\bSieg} \frac {d\bfbeta(u)} {(-\frac{i}{2} [ u_n + i \bfrho(z) ])^{n+s}}.
\end{align*}
Note that an element of $\bSieg$ has the form $(u^{\prime}, t + i|u^{\prime}|^2)$. A straightforward calculation yields
\begin{align*}
\int\limits_{\bSieg} \frac {d\bfbeta(u)} {(  u_n + i \bfrho(z) )^{n+s}}
=& \frac { 2\pi^{n-1}} {(n-2)!} \int\limits_{\bbC_{+}} \frac {(\ImPt \lambda)^{n-2}}
{(\lambda +i\bfrho(z))^{n+s}} dm_2(\lambda)\\
=& \frac { 2\pi^{n-1}} {(n-2)!}  \int\limits_{0}^{\infty} \bigg\{\int\limits_{-\infty}^{\infty}
\frac {dx} {(x+i(y+\bfrho(z)))^{n+s}} \bigg\} y^{n-2} dy.
\end{align*}
The inner integral vanishes by a contour integration argument, which proves the lemma.
\end{proof}

We refer to the following result as the cancellation property of $A_{\lambda}^1(\calU)$.

\begin{theorem}\label{thm:cancellation}
Let $\lambda>-1$ and $f\in A_{\lambda}^1(\calU)$. Then, for every $t>0$,
\[
\int\limits_{\bSieg} f(u+t\bfi) d\bfbeta(u)=0.
\]
Consequently,
\[
\int\limits_{\calU} f(z) dV_{\lambda}(z) =0.
\]
\end{theorem}

\begin{proof}
Let $f\in A_{\lambda}^1(\calU)$. By Lemma \ref{thm:DK},
\[
f(u+t\bfi) =  \int\limits_{\calU} \frac {f(w)}{\bfrho(u+t\bfi,w)^{n+1+\lambda}} dV_{\lambda}(w)
\]
for any $u\in b\calU$ and any $t>0$. Since $\bfrho(u+t\bfi,w)=\bfrho(u,w+t\bfi)$, we obtain
\begin{align*}\label{eqn:cancel1}
\int\limits_{\bSieg} f(u+t\bfi) d\bfbeta(u) =&  \int\limits_{\bSieg} \bigg\{ \int\limits_{\calU}
\frac{f(w)} {\bfrho(u,w+t\bfi)^{n+1+\lambda}} dV_{\lambda}(w) \bigg\} d\bfbeta(u)\\
=&  \int\limits_{\calU} \bigg\{ \int\limits_{\bSieg}
\frac {d\bfbeta(u)} {\bfrho(u,w+t\bfi)^{n+1+\lambda}} \bigg\} f(w) dV_{\lambda}(w) =0, \notag
\end{align*}
where the last equality follows from Lemma \ref{lem:bUrho=0}.
The interchange of the order of integration is justified as follows. 
Note that $\bfrho(w+t\bfi)=\bfrho(w)+t$.
By \eqref{eqn:keylem1},
\[
\int_{b\calU}
\frac{d\bfbeta(u)}
     {|\bfrho(u,w+t\bfi)|^{n+1+\lambda}}
=
\frac{4\pi^n\Gamma(1+\lambda)}
     {\Gamma^2\!\left(\frac{n+1+\lambda}{2}\right)}
\frac{1}{(\bfrho(w)+t)^{1+\lambda}}.
\]
Therefore,
\[
\int_{\calU}
\Bigg(
\int_{b\calU}
\frac{d\bfbeta(u)}
     {|\bfrho(u+t\bfi,w)|^{n+1+\lambda}}
\Bigg)
|f(w)|\, dV_\lambda(w)
\lesssim
t^{-(1+\lambda)}\,
\|f\|_{1,\lambda}<\infty.
\]
Finally, the identity
\[
\int_{\calU} f(z)\, dV_\lambda(z)
= c_{\lambda}
\int_0^\infty
\int_{b\calU} f(u+t\bfi)\, d\bfbeta(u)\, t^\lambda\, dt
\]
yields the second assertion.
\end{proof}

\section{The projection $\widetilde{\mathcal{P}_{\lambda}}$ from $L^{\infty}(\calU)$ into $\widetilde{\calB}$}

In this section we investigate the projection $\widetilde{\mathcal{P}_{\lambda}}$. We begin with an integral estimate of its kernel.

\begin{lemma}\label{lem:K}
Suppose $\lambda>-1$. Then, for every $z\in\calU$, one has
\begin{equation}\label{eqn:2normofK}
\int\limits_{\calU} \big|\widetilde{K_{\lambda}}(z,w)\big| dV_\lambda(w) \lesssim
1 + \log \frac {|\bfrho(z,\bfi)|^2} {\bfrho(z)}.
\end{equation}
Moreover,
\begin{equation}\label{eqn:doubleintofK}
\int\limits_{\calU} \int\limits_{\calU} \big|\widetilde{K_{\lambda}}(z,u) \widetilde{K_{\lambda}}(u,w)\big| dV_{\lambda}(w) dV_{\lambda}(u)
\lesssim \left\{1 + \log \frac {|\bfrho(z,\bfi)|^2} {\bfrho(z)}\right\}^2.
\end{equation}
\end{lemma}

\begin{proof}
Using the elementary inequality
\begin{equation}\label{eq:em:K}
\left|\frac {1}{A^{n+1+\lambda}} - \frac {1}{B^{n+1+\lambda}} \right| \leq (n+1) \left( \frac {|A-B| } {|A|^{n+1+\lambda} |B|} + \frac {|A-B| }{|A||B|^{n+1+\lambda}} \right),
\end{equation}
with $A=\bfrho(z,w)$ and $B=\bfrho(\bfi,w)$, we obtain
\begin{align*}
\big|\widetilde{K_{\lambda}}(z,w)\big|\leq (n+1) \left\{ \frac {|\bfrho(\bfi,w)-\bfrho(z,w)|}{|\bfrho(z,w)|^{n+1+\lambda} |\bfrho(\bfi,w)|}
+ \frac {|\bfrho(\bfi,w)-\bfrho(z,w)|}{|\bfrho(z,w)| |\bfrho(\bfi,w)|^{n+1+\lambda} }\right\}.
\end{align*}
Since $|z_n+i| > |z_n-i|$ and $|z_n+i|> |z^{\prime}|^2$ for all $z\in \calU$, we have
\begin{align*}
|\bfrho(&\bfi,w)-\bfrho(z,w)| \leq \frac {1}{2} |z_n - i| + |z^{\prime}||w^{\prime}|\\
\leq& \frac {1}{2} |z_n + i| + |z_n+i|^{1/2} |w_n+i|^{1/2}
= |\bfrho(z,\bfi)| + 2 |\bfrho(z,\bfi)|^{1/2}  |\bfrho(\bfi,w)|^{1/2}.
\end{align*}
Consequently,
\begin{equation*}\label{eqn:Kmajor}
\big|\widetilde{K_{\lambda}}(z,w)\big|  \lesssim \sum_{j=1}^4  K_j(z,w),
\end{equation*}
where 
\begin{align*}
K_1(z,w) = \frac {|\bfrho(z,\bfi)|}{|\bfrho(z,w)|^{n+1+\lambda} |\bfrho(\bfi,w)|}, & \quad
K_2(z,w) =  \frac {|\bfrho(z,\bfi)|^{1/2}}{|\bfrho(z,w)|^{n+1+\lambda} |\bfrho(\bfi,w)|^{1/2}},\\
K_3(z,w) =  \frac {|\bfrho(z,\bfi)|}{|\bfrho(z,w)| |\bfrho(\bfi,w)|^{n+1+\lambda}}, & \quad
K_4(z,w) =  \frac {|\bfrho(z,\bfi)|^{1/2} } {|\bfrho(z,w)|\, |\bfrho(\bfi,w)|^{n+1/2+\lambda}}.
\end{align*}

Let $z=\Phi(\eta)$ and $w=\Phi(\xi)$. For suitable exponents $a,b,s,k$, we obtain
\begin{align*}
\int\limits_{\calU} &\frac{\bfrho(w)^{s}} {|\bfrho(z,w)|^{a} |\bfrho(w,\bfi)|^{b}}
\left\{1+\log \frac {|\bfrho(w,\bfi)|^2}{\bfrho(w)}\right\}^k dV(w)\\
&= 4|1+\eta_n|^{a} \int\limits_{\ball} \frac{(1-|\xi|^2)^{s}} {|1-\eta\cdot\overline{\xi}|^{a}
|1+\xi_n|^{2(n+1+s)-a-b}} \left(\log \frac {e}{1-|\xi|^2}\right)^k dV(\xi).
\end{align*}
By \cite[Theorem 3.1]{ZLSG18}, it follows that
\[
\int\limits_{\calU}  K_j (z,w)dV_\lambda(w) \lesssim
1 + \log \frac {|\bfrho(z,\bfi)|^2} {\bfrho(z)},\qquad j=1,\ldots,4,
\]
which proves \eqref{eqn:2normofK}.
To obtain \eqref{eqn:doubleintofK}, it suffices to show that
\begin{align*}
\int\limits_{\calU} K_j(z,u) \left\{1 + \log \frac {|\bfrho(u,\bfi)|^2} {\bfrho(u)}\right\} dV_\lambda(u)
\lesssim
\left\{1 + \log \frac {|\bfrho(z,\bfi)|^2} {\bfrho(z)}\right\}^2,
\end{align*}
which again follows from \cite[Theorem 3.1]{ZLSG18}.
\end{proof}

\begin{lemma}\label{lem:Kreproduce}
Let $\lambda>-1$. Then, for all $z,w\in\calU$,
\[
\widetilde{\mathcal{P}_\lambda}(\widetilde{K_\lambda}(\cdot,w))(z)=\widetilde{K_\lambda}(z,w).
\]
\end{lemma}

\begin{proof}
Fix $z,w\in \calU$. By definition,
\begin{align*}
&\widetilde{\mathcal{P}_\lambda}(\widetilde{K_\lambda}(\cdot,w))(z)
=\int\limits_{\calU} \widetilde{K_{\lambda}} (z,u) \widetilde{K_{\lambda}}(u,w) dV_\lambda(u)\\
&=\int\limits_{\calU} \widetilde{K_{\lambda}} (z,u) K_{\lambda}(u,w) dV_\lambda(u)
-\int\limits_{\calU} \widetilde{K_{\lambda}} (z,u) K_{\lambda}(\bfi,w) dV_\lambda(u)\\
&=\overline{\int\limits_{\calU} \overline{\widetilde{K_{\lambda}} (z,u)} K_{\lambda}(w,u) dV_\lambda(u)}
-K_{\lambda}(\bfi,w) \overline{\int\limits_{\calU} \overline{\widetilde{K_{\lambda}} (z,u)} dV_\lambda(u)}.
\end{align*}
Using \eqref{eqn:2normofK}, we see that $\overline{\widetilde{K_\lambda}(z,\cdot)} \in A_\lambda^1(\calU)$. 
By Lemma \ref{thm:DK},
\[
\int\limits_{\calU} \overline{\widetilde{K_{\lambda}} (z,u)} K_{\lambda}(w,u) dV_\lambda(u)
= \overline{\widetilde{K_\lambda}(z,w)},
\]
and by the cancellation property of $A_\lambda^1(\calU)$,
\[
\int\limits_{\calU} \overline{\widetilde{K_{\lambda}} (z,u)} dV_\lambda(u)=0.
\]
The conclusion follows by taking complex conjugates.
\end{proof}

\begin{lemma}\label{lem:P_lambda}
Let $\lambda>-1$ and let $\alpha\in \mathbb{N}_0^n$ with $|\alpha|\geq 1$. Set $f=\widetilde{\mathcal{P}_\lambda}(g)$. 
Then $\bfrho^{\langle \alpha \rangle} \bfL^{\alpha} f\in L^{\infty}(\calU)$ if $g\in L^{\infty}(\calU)$,
 $\bfrho^{\langle \alpha \rangle} \bfL^{\alpha} f\in C_0(\calU)$ if  $g\in C_0(\calU)$.
\end{lemma}
\begin{proof}
Assume first that $g\in L^{\infty}(\calU)$. By \eqref{eqn:2normofK}, the function $f$ is holomorphic in $\calU$.  
A direct computation gives
\[
\bfL^{\alpha} \left\{\frac {1}{\bfrho(z,w)^{n+1+\lambda}} \right\}  = C(n,\lambda,\alpha)
\frac {\left(\overline{{z}^{\prime}}-\overline{{w}^{\prime}}\right)^{\alpha^{\prime}}}
{\bfrho(z,w)^{n+1+\lambda+|\alpha|}},
\]
where $C(n,\lambda,\alpha)$ is a constant depending on $n$, $\lambda$ and $\alpha$. 
Since $\big|(z^{\prime}-w^{\prime})^{\alpha^{\prime}}\big| \lesssim |\bfrho(z,w)|^{\frac {|\alpha^{\prime}|}{2}}$ by \cite[Lemma 3.2]{LSX22} and $|\alpha|=\langle \alpha \rangle+|\alpha^{\prime}|/2$, we obtain
\begin{equation*}\label{eq:ptwsest}
\frac{ \big|(z^{\prime}-w^{\prime})^{\alpha^{\prime}}\big|}{|\bfrho(z,w)|^{n+1+\lambda+|\alpha|}} 
\lesssim \frac{1}{|\bfrho(z,w)|^{n+1+\lambda+\langle \alpha \rangle}}.
\end{equation*}
Therefore,
\begin{align*}
\big|\bfrho(z)^{\langle \alpha \rangle} \bfL^{\alpha} f(z)\big|
\lesssim& \bfrho(z)^{\langle \alpha \rangle}
\int\limits_{\calU} \frac { \big|(z^{\prime}-w^{\prime})^{\alpha^{\prime}}\big| \bfrho(w)^{\lambda}}
{|\bfrho(z,w)|^{n+1+\lambda+|\alpha|}} |g(w)| dV(w)\\
\lesssim&  \bfrho(z)^{\langle \alpha \rangle}
\int\limits_{\calU} \frac { \bfrho(w)^{\lambda}}
{|\bfrho(z,w)|^{n+1+\lambda+\langle \alpha \rangle}} |g(w)| dV(w).
\end{align*}
An application of \eqref{eqn:keylem} shows that $\bfrho^{\langle \alpha \rangle} \bfL^{\alpha} f\in L^{\infty}(\calU)$. 

Now assume $g\in C_0(\calU)$.
Given $\epsilon>0$, there exists $R>0$ such that
\[
\sup \left\{ |g(w)|: w\in \calU\setminus \overline{D(\bfi, R)} \right\} < \epsilon,
\]
where $D(\bfi, R):=\{w\in \calU: \beta(w,\bfi)< R\}$.
We decompose
\begin{align*}
|\bfrho(z)^{\langle\alpha\rangle}\bfL^\alpha f(z)|
&\lesssim
\underbrace{
\bfrho(z)^{\langle\alpha\rangle}
\int_{\calU\setminus\overline{D(\bfi,R)}}
\frac{\bfrho(w)^\lambda}
     {|\bfrho(z,w)|^{n+1+\lambda+\langle\alpha\rangle}}
\, dV(w)}_{=:I_1(z)}\,\varepsilon\\
&\quad+
\underbrace{
\bfrho(z)^{\langle\alpha\rangle}
\int_{\overline{D(\bfi,R)}}
\frac{\bfrho(w)^\lambda}
     {|\bfrho(z,w)|^{n+1+\lambda+\langle\alpha\rangle}}
\, dV(w)}_{=:I_2(z)}
\|g\|_\infty .
\end{align*}
By \eqref{eqn:keylem}, $I_1(z)$ is bounded.
For $w\in D(\bfi,R)$, Lemma \ref{blem} yields
\[
\frac{ \bfrho(w)^{\lambda}}{|\bfrho(z,w)|^{n+1+\lambda+\langle \alpha \rangle}}
\lesssim \frac{1} {|\bfrho(z,\bfi)|^{n+1+\lambda+\langle \alpha \rangle}}.
\]
Hence,
\[
I_2(z) \lesssim \frac{\bfrho(z)^{\langle \alpha \rangle} } {|\bfrho(z,\bfi)|^{n+1+\lambda+\langle \alpha \rangle}}
\lesssim \min\left\{\bfrho(z)^{\langle \alpha \rangle},\frac{{|z|}^{\langle \alpha \rangle}}{(|z|^2+1)^{(n+1+\lambda+\langle \alpha \rangle)/2}}\right\},
\]
where we use \eqref{eqn:rho(z,i)} in the last step.
It implies that $I_2(z)\to 0$ as $z\to\partial\widehat{\calU}$.
Therefore, $\bfrho^{\langle \alpha \rangle} \bfL^{\alpha} f \in C_0(\calU)$.
\end{proof}

\begin{corollary}\label{cor:P_lambdainto}
Let $\lambda>-1$.  Then $\widetilde{\mathcal{P}_\lambda}$ maps $L^\infty(\calU)$ boundedly into $\widetilde{\calB}$,
and maps $C_0(\calU)$ boundedly into $\widetilde{\calB}_0$.
\end{corollary}
\begin{proof}
Recall that 
\[
|\widetilde{\nabla}f|^2=2\sum_{j=1}^{n-1}|\bfrho^{1/2}\bfL_j f|^2+8|\bfrho\bfL_n f|^2.
\]
An application of Lemma \ref{lem:P_lambda} to all multi-indices $\alpha$ with $|\alpha|=1$ gives the result.
\end{proof}

\section{Integral representation}

For any $t>0$, let $\mathcal{A}^{-t}$ denote the space of holomorphic functions $f$ on $\calU$ such that 
\[
\sup_{z\in\calU} \bfrho(z)^t |f(z)|<\infty.
\]
These spaces are usually called Korenblum spaces, following the terminology used for the unit disk (c.f. \cite{HKZ20}).
Clearly, $\mathcal{A}^{-t}\not\subset A_\lambda^p(\calU)$.
Nevertheless, functions in $\mathcal{A}^{-t}$ admit the same integral representation as Bergman functions.

\begin{lemma}\label{lem:Bek83}
Let $t>0$ and $f\in\mathcal{A}^{-t}$. Then $f=\mathcal{P}_{\lambda}f$ for any $\lambda> t-1$.
\end{lemma}
\begin{proof}
Fix $\lambda> t-1$. By \eqref{eqn:keylem}, the integral 
\[
\int\limits_{\calU} \frac{\bfrho(w)^{\lambda}}{|\bfrho(z,w)|^{n+1+\lambda}} |f(w)| dV(w)
\]
converges for every $z\in\calU$. 
Let $z=\Phi(\xi)$ and $w=\Phi(\eta)$. Then
\begin{align*}
\mathcal{P}_{\lambda}f(z)&=c_{\lambda} \int\limits_{\calU} \frac{\bfrho(w)^{\lambda}}{\bfrho(z,w)^{n+1+\lambda}} f(w) dV(w)\\
&=4c_{\lambda} \int\limits_{\ball} \frac{\left(\frac{1-|\eta|^2}{|1+\eta_n|^2}\right)^{\lambda}}{\left(\frac{1-\xi\cdot\overline{\eta}}{(1+\xi_n)(1+\overline{\eta}_n)}\right)^{n+1+\lambda}}
 f(\Phi(\eta)) \frac{1}{|1+\eta_n|^{2(n+1)}} dV(\eta)\\
 &=(1+\xi_n)^{n+1+\lambda}
 \int\limits_{\ball} \frac{1}{(1-\xi\cdot\overline{\eta})^{n+1+\lambda}} 
 \frac{f(\Phi(\eta))}{(1+\eta_n)^{n+1+\lambda}} dv_{\lambda}(\eta)\\
 &:=(1+\xi_n)^{n+1+\lambda}
 \int\limits_{\ball} \frac{F(\eta)}{(1-\xi\cdot\overline{\eta})^{n+1+\lambda}}  dv_{\lambda}(\eta),
\end{align*}
where $dv_{\lambda}(\xi)= 4 c_{\lambda}  (1-|\xi|^2)^{\lambda} dV(\xi)$ is a probability measure on $\ball$.
A direct estimate gives
\[
 \int\limits_{\ball} |F(\eta)| dv_{\lambda}(\eta) \lesssim  \int\limits_{\ball} \frac{\bfrho(\Phi(\eta))^{-t}}{|1+\eta_n|^{n+1+\lambda}}dv_{\lambda}(\eta) \lesssim \int\limits_{\ball}\frac{(1-|\eta|^2)^{\lambda-t}}{|1+\eta_n|^{n+1+\lambda-2t}} dv(\eta)<\infty.
\]
Then by the reproducing formula of weighted Bergman spaces on $\ball$, we obtain
\[
\mathcal{P}_{\lambda}f(z)= (1+\xi_n)^{n+1+\lambda} F(\xi) =f(z),
\]
as desired.
\end{proof}

\begin{lemma}\label{lem:uniq_n}
Let $f\in \widetilde{\calB}$. If $\bfL_n f\equiv 0$ then $f\equiv 0$.
\end{lemma}
\begin{proof}
 Suppose $f\in \widetilde{\calB}$ and $\bfL_n f\equiv 0$. Then $f$ is independent of the $n$-th coordinate, and hence the same is true for each $\partial_j f$.
Consequently,
\begin{equation*}\label{eqn:partial_nf}
\big|\widetilde{\nabla} f(z)\big|^2 = 2 \left(\ImPt z_n-|z^{\prime}|^2\right)
\sum_{j=1}^{n-1} \left|\partial_j f(z^{\prime})\right|^2.
\end{equation*}
Fix $z^{\prime}$ and let  $\ImPt z_n\to\infty$.
Then the right-hand side of the above equality tends to infinity, whereas the left-hand side remains bounded.
This is possible only if $\partial_j f \equiv 0$ for all $j\in \{1,\ldots,n-1\}$. 
Hence $f$ is constant. Since $f(\bfi)=0$, we conclude that $f\equiv 0$.
\end{proof}


\begin{proof}[Proof of Theorem \ref{thm:reproducingfml}]
Let $f\in \widetilde{\calB}$ and fix $\lambda>-1$. Set
\[
b_N:=\frac{(-2i)^N\Gamma(1+\lambda)}{\Gamma(1+\lambda+N)}.
\]
We divide the argument into three cases.

\medskip
\noindent\textit{Case I: $N=1$.}
In this case the desired identity takes the form
\begin{equation}\label{eqn:reproducingfml1}
f(z) = b_1 \int\limits_{\calU} \widetilde{K_\lambda}(z,w) \bfrho(w) \bfL_n f(w) dV_\lambda(w).
\end{equation}
To prove this, put
\[
g(z) = f(z) - b_1 \int\limits_{\calU} \widetilde{K_\lambda}(z,w) \bfrho(w)  \bfL_n f(w) dV_\lambda(w).
\]
Since $f\in \widetilde{\calB}$, we have $\bfrho\bfL_n f\in L^{\infty}(\calU)$,
and therefore Corollary \ref{cor:P_lambdainto} implies that $g\in \widetilde{\calB}$.
A direct computation yields
\begin{equation*}\label{eqn:partg}
\bfL_n g (z) = \bfL_n f(z) - \mathcal{P}_{\lambda+1}(\bfL_nf)(z) = 0,
\end{equation*}
where the last equality follows from the fact that $\bfL_nf\in\mathcal{A}^{-1}$ together with  Lemma \ref{lem:Bek83}.
Lemma \ref{lem:uniq_n} now shows that $g\equiv 0$, which proves \eqref{eqn:reproducingfml1}.

\medskip
\noindent\textit{Case II: $N=0$.}
Using \eqref{eqn:reproducingfml1} we obtain
\begin{align*}
&\widetilde{\mathcal{P}_{\lambda}}f(z)=\int\limits_{\calU} \widetilde{K_\lambda}(z,u) f(u)dV_{\lambda}(u)\\
&=b_1 \int\limits_{\calU} \widetilde{K_\lambda}(z,u)
\bigg\{\int\limits_{\calU}  \widetilde{K_\lambda}(u,w) \bfrho(w) \bfL_n f(w) dV_\lambda(w)\bigg\} dV_{\lambda}(u).
\end{align*}
By \eqref{eqn:doubleintofK}, 
\begin{align*}
\int\limits_{\calU} \int\limits_{\calU}& \big|\widetilde{K_{\lambda}}(z,u) \widetilde{K_{\lambda}}(u,w)\bfrho(w) \bfL_n f(w) \big| dV_{\lambda}(w) dV_{\lambda}(u)\\
&\lesssim \|f\|_{\calB} \left\{1 + \log \frac {|\bfrho(z,\bfi)|^2} {\bfrho(z)}\right\}^2<\infty,
\end{align*}
so Fubini's theorem applies. Hence
\begin{align*}
\widetilde{\mathcal{P}_\lambda} f(z) =& b_1\int\limits_{\calU} \bfrho(w) \bfL_n f(w) \bigg\{\int\limits_{\calU} \widetilde{K_\lambda}(z,u) \widetilde{K_\lambda}(u,w) dV_{\lambda}(u) \bigg\} dV_{\lambda}(w)\\
=& b_1 \int\limits_{\calU} \bfrho(w) \bfL_n f(w)  \widetilde{K_\lambda}(z,w) dV_\lambda(w)= f(z),
\end{align*}
where we used Lemma \ref{lem:Kreproduce} and again \eqref{eqn:reproducingfml1}.

\medskip
\noindent\textit{Case III: $N\geq 2$.}
Applying $\bfL_n^N$ to both sides of \eqref{eqn:reproducingfml1} and multiplying by $\bfrho(z)^N$,
we have
\begin{equation*}\label{eqn:N-Df}
\bfrho(z)^N \bfL_n^N f(z) = \frac{b_1}{b_N}  c_{\lambda+N}
\int\limits_{\calU} \frac{\bfrho(z)^N} {\bfrho(z,w)^{n+1+\lambda+N}} \bfrho(w)^{1+\lambda} \bfL_n f(w) dV(w).
\end{equation*}
By Fubini's theorem,
\begin{align*}
&\widetilde{\mathcal{P}_\lambda} (\bfrho^N \bfL_n^N f) (z)\\
=& \frac{b_1}{b_N}  c_{\lambda+N} \int\limits_{\calU}  \widetilde{K_\lambda}(z,u) 
\bigg\{  \int\limits_{\calU} \frac {\bfrho(u)^N} {\bfrho(u,w)^{n+1+\lambda+N}} \bfrho(w)^{1+\lambda} \bfL_nf(w) dV(w) \bigg\}dV_\lambda(u)\\
=&\frac{b_1}{b_N} c_\lambda \int\limits_{\calU} \bfrho(w)^{1+\lambda} \bfL_nf(w)
\bigg\{ \overline{ c_{\lambda+N} \int\limits_{\calU} \frac {\bfrho(u)^{\lambda+N}} {\bfrho(w,u)^{n+1+\lambda+N}} \overline{\widetilde{K_\lambda}(z,u)} dV(u) \bigg\} }dV(w)\\
=& \frac{b_1}{b_N} c_\lambda \int\limits_{\calU} \bfrho(w)^{1+\lambda} \bfL_nf(w) \widetilde{K_\lambda}(z,w) dV(w) =  \frac{f(z)}{b_N},
\end{align*}
where we used that $\overline{\widetilde{K_\lambda}(z,\cdot)}\in A_\lambda^1(\calU)$ and Lemma \ref{thm:DK}, and again \eqref{eqn:reproducingfml1}.
The interchange of integral order is justified as follows. By \eqref{eqn:keylem} and \eqref{eqn:2normofK},
\begin{align*}
\int\limits_{\calU}& \int\limits_{\calU}   \left| \widetilde{K_\lambda}(z,u) \frac {\bfrho(u)^{N}} {\bfrho(u,w)^{n+1+\lambda+N}}  \bfrho(w)^{1+\lambda} \bfL_nf(w)\right| dV(w)dV_\lambda(u)\\
&\lesssim \|f\|_{\calB} \int\limits_{\calU} \big|\widetilde{K_\lambda}(z,u)\big| \bigg\{\bfrho(u)^{N}\int\limits_{\calU} \frac {\bfrho(w)^{\lambda}dV(w)} {|\bfrho(u,w)|^{n+1+\lambda+N}} \bigg\}  dV_\lambda(u)\\
&\lesssim \|f\|_{\calB} \int\limits_{\calU} \big|\widetilde{K_\lambda}(z,u)\big| dV_\lambda(u) \lesssim  \|f\|_{\calB}\left\{1 + \log \frac {|\bfrho(z,\bfi)|^2} {\bfrho(z)}\right\}< \infty
\end{align*}
for every $z\in\calU$. This completes the proof of the theorem.
 \end{proof}

\section{Norm equivalence}

\begin{proposition}\label{lem:bddnsofder4bloch}
Let $\alpha\in\bbN_0^n$ with $|\alpha|\ge1$. The map $f\mapsto \bfrho^{\langle \alpha \rangle} \bfL^{\alpha} f$ is  bounded  from $\widetilde{\calB}$ into $L^{\infty}(\calU)$, and bounded from $\widetilde{\calB}_0$ into $C_0(\calU)$.
\end{proposition}
\begin{proof}
Let $f\in \widetilde{\calB}$ (resp. $\widetilde{\calB}_0$). 
Then $\bfrho \bfL_n f\in L^{\infty}(\calU)$ (resp. $C_0(\calU)$) and $\|\bfrho\calL_nf\|_\infty\lesssim \|f\|_{\calB}$. 
Theorem \ref{thm:reproducingfml} tells us that 
$f=  b_1  \widetilde{\mathcal{P}_{\lambda}} (\bfrho \bfL_n f)$
for any $\lambda>-1$. This together with Lemma \ref{lem:P_lambda} gives the desired boundedness.
\end{proof}

 \begin{proposition}\label{cor:BPL}
Let $\lambda>-1$. The operator $\widetilde{\mathcal{P}_\lambda}$ is bounded from $L^{\infty}(\calU)$ onto $\widetilde{\calB}$, and from $C_0(\calU)$ onto $\widetilde{\calB}_0$.
\end{proposition}
\begin{proof}
The boundedness and the ``into'' part were already established in Corollary \ref{cor:P_lambdainto}.
The ``onto'' part follows directly from Theorem \ref{thm:reproducingfml} with $N=1$.
\end{proof}

Now, we are ready to prove Theorem \ref{thm:bloch norm equiv}.

\begin{proof}[Proof of Theorem \ref{thm:bloch norm equiv}]
It suffices to show that 
\begin{equation*}\label{eqn:thm:bloch norm equiv}
\sum_{|\alpha | = N} \| \bfrho^{\langle \alpha \rangle} \bfL^{\alpha} f\|_{\infty}
\lesssim \|f\|_{\calB} \lesssim \|\bfrho^N \bfL_n^N f\|_{\infty}.
\end{equation*}
The first inequality is an immediate consequence of Proposition \ref{lem:bddnsofder4bloch}.
To prove the second one, note from Theorem \ref{thm:reproducingfml} that
$f=  b_N \widetilde{\mathcal{P}_{\lambda}} (\bfrho^N \bfL_n^N f)$,
together with Proposition \ref{cor:BPL} yields the desired result.
\end{proof}

An application of Theorem \ref{thm:bloch norm equiv} could generalize Lemma \ref{lem:uniq_n} to all cases.
We refer to it as the uniqueness property for functions in $\widetilde{\calB}$. It shows how differently Bloch functions on $\calU$ behave when compared with those on the unit ball. 

\begin{corollary}\label{thm:uniquess}
Suppose $f\in\widetilde{\calB}$. If $\bfL^{\alpha} f\equiv 0$ for some $\alpha\in \mathbb{N}_0^n$, then $f\equiv 0$.
\end{corollary}

\begin{proof}
The argument is divided into two steps.

\medskip
\noindent\textit{Step 1: The case $\alpha=(0',N)$.}
Assume that $\bfL_n^N f\equiv0$. Then $f$ is necessarily a polynomial of degree at most $N-1$ in the variable $z_n$, namely
\[
f(z)=f_{N-1}(z^{\prime}) z_n^{N-1}+f_{N-2}(z^{\prime}) z_n^{N-2}+\cdots+ f_0(z^{\prime}),
\]
where each $f_j$ is holomorphic in $z'$. 
Differentiating, we obtain
\[
\bfrho(z)^{N-1}\bfL_n^{N-1} f(z) = (N-1)! \bfrho(z)^{N-1} f_{N-1}(z^{\prime}).
\]
By Theorem \ref{thm:bloch norm equiv}, the left-hand side is bounded on $\calU$, whereas the right-hand side tends to $+\infty$ whenever $z'$ is fixed and $\ImPt z_n\to+\infty$, unless $f_{N-1}\equiv0$. 
Iterating this argument yields $f_{N-2}=\cdots=f_1\equiv0$, so that $f(z)=f_0(z')$.  
Lemma~\ref{lem:uniq_n} then forces $f\equiv0$.


\medskip
\noindent\textit{Step 2: The general case.}
Assume now that $\bfL^\alpha f\equiv0$ for some multi-index $\alpha\in\bbN_0^n$.  
Using the expansion
\[
\bfL^{\alpha} f ~=~ \sum_{\gamma^{\prime} \leq \alpha^{\prime}} \binom{\alpha^{\prime}}{\gamma^{\prime}}
(2i\bar{z}^{\prime})^{\gamma^{\prime}}
\partial^{\alpha-(\gamma^{\prime}, -|\gamma^{\prime}|)}f,
\]
we may regard the left-hand side as a polynomial in $\bar{z}^{\prime}$. 
Since this polynomial vanishes identically, all its coefficients must be zero; hence
\[
\partial^{\alpha - (\gamma', -|\gamma'|)} f \equiv 0
\qquad\text{for all }\gamma'\le \alpha'.
\]
In particular, taking $\gamma'=\alpha'$ yields
\[
\bfL_n^{|\alpha|} f \equiv 0.
\]
By Step 1, this implies $f\equiv0$.  
The proof of the theorem is complete.
\end{proof}

To address potential concerns, readers might note that the parallels between Theorems \ref{thm:reproducingfml}, \ref{thm:bloch norm equiv}, and Corollary \ref{thm:uniquess}, and those in our previous work \cite{LSX22} on weighted Bergman space, possibly questioning whether this constitutes merely a straightforward generalization or follows the same technical path.
This is not the case.
In the case of weighted Bergman space, the uniqueness property served as the starting point, from which the integral representation and norm equivalences were subsequently derived.
While for Bloch-type space, the logical direction is reversed. The integral representation with respect to $\widetilde{\mathcal{P}_{\lambda}}$ constitutes the central ingredient, the norm equivalence follow from the boundedness and integral representation of $\widetilde{\mathcal{P}_{\lambda}}$, and only thereafter the uniqueness property is obtained.

\section{$\widetilde{\calB}_0$ as the predual of $A_{\lambda}^1(\calU)$}

We denote by $M(\calU)$ the space of complex Radon measures on $\calU$, and for $\mu\in M(\calU)$ we write
\[
\|\mu\|=|\mu|(\calU),
\]
where $|\mu|$ denotes the total variation of $\mu$.

\begin{lemma}\label{lem:mu}
Let $\lambda>-1$, $\mu\in M(\calU)$, and define
\[
g(w):=\int\limits_{\calU} \frac{\bfrho(z)}{\bfrho(w,z)^{n+2+\lambda}} d\mu(z), \quad w\in \calU.
\]
Then $g\in A_\lambda^1(\calU)$ and $\|g\|_{1,\lambda}\lesssim \|\mu\|$.
\end{lemma}
\begin{proof}
By \eqref{eqn:keylem}, there is a positive constant $C$ such that
\begin{eqnarray*}
\|g\|_{1,\lambda} &=& \int\limits_{\calU}\bigg|\int\limits_{\calU} \frac {\bfrho(z)}{\bfrho(w,z)^{n+2+\lambda}} d\mu(z)\bigg| dV_\lambda(w)\\
&\leq& \int\limits_{\calU} \bigg\{\int\limits_{\calU} \frac {\bfrho(w)^\lambda dV(w)}{|\bfrho(w,z)|^{n+2+\lambda}} \bigg\} \bfrho(z) d|\mu|(z)
\leq C\|\mu\|,
\end{eqnarray*}
which proves the lemma.
\end{proof}

\begin{proof}[Proof of Theorem \ref{thm:predual}]
For each $g\in A_\lambda^1(\calU)$, it is obvious that the map
\[
f \longmapsto  \int\limits_{\calU} (\bfrho \bfL_n f )\overline{g} dV_\lambda,\quad f\in\widetilde{\calB}_0,
\]
defines a bounded linear functional on $\widetilde{\calB}_0$,  with norm $\lesssim\|g\|_{1,\lambda}$.

Now let $\Lambda$ be a bounded linear functional on $\widetilde{\calB}_0$.  
By Theorem \ref{thm:bloch norm equiv}, the map
\begin{equation*}\label{eq:map_L}
\mathcal{F}:\widetilde{\calB}_0 \to C_0(\calU),\ \ \ f \mapsto  \bfrho \bfL_n f,
\end{equation*}
is an embedding whose range is a closed subspace of $C_0(\calU)$.
Therefore, $\Lambda\circ \mathcal{F}^{-1}$
is a bounded linear functional on $\mathcal{F}(\widetilde{\calB}_0)$.
By Hahn-Banach extension theorem and Riesz representation theorem, there exists $\mu\in M(\calU)$ such that
$ \|\mu\|=\|\Lambda\circ \mathcal{F}^{-1}\|$ and
\[
\Lambda (f) = \Lambda\circ \mathcal{F}^{-1} (\mathcal{F} f) = \int\limits_{\calU} \mathcal{F} f d\mu
= \int\limits_{\calU} \bfrho(z) \bfL_n f(z) d\mu(z)
\]
for all $f\in \widetilde{\calB}_0$. Given $\beta>0$, it follows from Lemma \ref{lem:Bek83} that
\begin{align*}
\Lambda (f)&= \int\limits_{\calU} \bfrho(z) \mathcal{P}_\beta(\bfL_nf)(z) d\mu(z)\\
&= \int\limits_{\calU} \bfrho(z) \bigg\{ c_\beta
\int\limits_{\calU} \frac{\bfrho(w)^\beta}{\bfrho(z,w)^{n+1+\beta}} \bfL_n f(w) dV(w) \bigg\} d\mu(z).
\end{align*}
Using \eqref{eqn:keylem}, we have
\begin{align*}
\int\limits_{\calU} \int\limits_{\calU} & \frac {\bfrho(w)^\beta}{|\bfrho(z,w)|^{n+1+\beta}} |\bfL_n f(w)| \bfrho(z) dV(w) d|\mu|(z)\\
\lesssim &  \|f\|_{\calB} \int\limits_{\calU} \bfrho(z) \bigg\{\int\limits_{\calU}  \frac {\bfrho(w)^{\beta-1} dV(w)} {|\bfrho(z,w)|^{n+1+\beta}}
\bigg\} d|\mu|(w) 
\lesssim  \|f\|_{\calB} \|\mu\|.
\end{align*}
Thus Fubini's theorem applies, and we obtain
\begin{align*}
\Lambda(f) =&  \int\limits_{\calU} \bfrho(w)^\beta \bfL_n f(w) \bigg\{ c_\beta \int\limits_{\calU} \frac {\bfrho(z)}{\bfrho(z,w)^{n+1+\beta}} d\mu(z)\bigg\}  dV(w)\\
=&  \int\limits_{\calU} \bfrho(w) \bfL_n f(w) \overline{g(w)} dV_{\lambda}(w),
\end{align*}
where $\lambda=\beta-1>-1$ and 
\[
g(w) = c_\beta \int\limits_{\calU} \frac {\bfrho(z)}{\bfrho(w,z)^{n+2+\lambda}} d\overline{\mu}(z).
\]
By Lemma \ref{lem:mu}, $g\in A_\lambda^1(\calU)$, and moreover
\[
\|g\|_{1,\lambda}\lesssim \|\mu\|=\|\Lambda\circ \mathcal{F}^{-1}\| \lesssim \|\Lambda\|.
\]
This completes the proof of the theorem.
\end{proof}

\end{document}